\documentclass{article}
\usepackage{indentfirst,overpic}
\usepackage{amsthm,afterpage,float}
\usepackage{amsfonts}
\usepackage{amssymb}
\usepackage{overpic}

\newtheorem{remark}{Remark}

\newcommand{\COLORON}{1}
\newcommand{\NOTESON}{0}
\newcommand{\Debug}{0}

\newcommand{\forb}[1]{\mathrm{Forb}(#1)}
\newcommand{\ex}[1]{\mathrm{Ex}(#1)}
\newcommand{\prl}[1]{#1^{\obslash}}
\newcommand{\uop}{\ensuremath{U}-\OP}

\newcommand{\pln}{\ensuremath{\mathrm{Planar}}}

\newcommand{\plV}{\ensuremath{\mathrm{Planar_V}}}
\newcommand{\frs}{\mathcal{F}}
\newcommand{\frsV}{\frs_\mathrm{V}}
\newcommand{\frsE}{\frs_\mathrm{E}}
\newcommand{\frsCE}{\frs_\mathrm{{/E}}}
\newcommand{\rmv}[1]{\ensuremath{#1_{\mathrm{V}}}}
\newcommand{\rme}[1]{\ensuremath{#1_{\mathrm{E}}}}
\newcommand{\rmce}[1]{\ensuremath{#1_{\mathrm{/E}}}}
\newcommand{\rmece}[1]{\ensuremath{#1_{\mathrm{E/E}}}}
\newcommand{\ope}{\rme{OP}}
\newcommand{\OuPl}{\ensuremath{{\mathrm{OP}}}}
\newcommand{\SU}{\ensuremath{\Sig_\bullet}}
\newcommand{\dbar}[1]{\ensuremath{\overline{\overline{#1}}}}

\newcommand{\cof}{co-finite}
\newcommand{\Cof}{Co-finite}
\newcommand{\uncof}{UNCOF}

\newcommand{\mm}{marked minor}
\newcommand{\umm}{\ensuremath{U}-marked minor}
\newcommand{\Sig}{\ensuremath{\Sigma}}
\newcommand{\vapf}{VAP-free}
\newcommand{\Pof}{Proof of }

\newcommand{\omdot}{\omega \cdot}

\usepackage[usenames]{color} 
\usepackage{amsthm,amssymb,amsmath,bbm,enumerate,graphicx,epsf,stmaryrd,accents}
\usepackage[bookmarks, colorlinks=false, breaklinks=true]{hyperref} 

\usepackage{authblk}

\newcommand{\comment}[1]{}
\newcommand{\COMMENT}[1]{}

\definecolor{darkgray}{rgb}{0.3,0.3,0.3}
\newcommand{\defi}[1]{{\color{darkgray}\emph{#1}}}

\newcommand{\acknowledgement}{\section*{Acknowledgement}}

\comment{
	\begin{lemma}\label{}	
\end{lemma}
\begin{proof}

\end{proof}

\begin{theorem}\label{}
\end{theorem} 
\begin{proof} 	

\end{proof}

}

\newtheorem{proposition}{Proposition}[section]
\newtheorem{definition}[proposition]{Definition}
\newtheorem{theorem}[proposition]{Theorem}
\newtheorem{corollary}[proposition]{Corollary}

\newtheorem{lemma}[proposition]{Lemma}

\newtheorem{conjecture}{{Conjecture}}[section]

\newtheorem{problem}[conjecture]{{Problem}}

\newtheorem{question}[conjecture]{{Question}}

\newtheorem{examp}[proposition]{Example}

\newcommand{\FIG}{0}

\ifnum \NOTESON = 1 \newcommand{\note}[1]{ 

\hspace*{-30pt}
	{\color{blue}  NOTE: \color{Turquoise}{\small  \tt \begin{minipage}[c]{1.1\textwidth}  #1 \end{minipage} \ignorespacesafterend }} 
	
	}
\else \newcommand{\note}[1]{} \fi

\newcommand{\afsubm}[1]{ \ifnum \Debug = 1 {\mymargin{#1}}
\fi} 

\ifnum \Debug = 1 
\else  \fi

\ifnum \FIG = 1 \newcommand{\fig}[1]{Figure ``{#1}''}
\else \newcommand{\fig}[1]{Figure~\ref{#1}} \fi

\ifnum \FIG = 1 
\else  \fi

\ifnum \Debug = 1 \usepackage[notref,notcite]{showkeys}
\fi

\ifnum \COLORON = 0 \renewcommand{\color}[1]{}
\fi

\newcommand{\N}{\ensuremath{\mathbb N}}
\newcommand{\R}{\ensuremath{\mathbb R}}

\newcommand{\BS}{\ensuremath{\mathbb S}}

\newcommand{\cb}{\ensuremath{\mathcal B}}
\newcommand{\cc}{\ensuremath{\mathcal C}}

\newcommand{\ci}{\ensuremath{\mathcal I}}

\newcommand{\cp}{\ensuremath{\mathcal P}}

\newcommand{\sm}{\backslash}

\newcommand{\cls}[1]{\ensuremath{\overline{#1}}}

\makeatletter
\DeclareRobustCommand{\cev}[1]{%
  \mathpalette\do@cev{#1}%
}
\newcommand{\do@cev}[2]{%
  \fix@cev{#1}{+}%
  \reflectbox{$\m@th#1\vec{\reflectbox{$\fix@cev{#1}{-}\m@th#1#2\fix@cev{#1}{+}$}}$}%
  \fix@cev{#1}{-}%
}
\newcommand{\fix@cev}[2]{%
  \ifx#1\displaystyle
    \mkern#23mu
  \else
    \ifx#1\textstyle
      \mkern#23mu
    \else
      \ifx#1\scriptstyle
        \mkern#22mu
      \else
        \mkern#22mu
      \fi
    \fi
  \fi
}

\makeatother

\newcommand{\nin}{\ensuremath{{n\in\N}}}

\newcommand{\pth}[2]{\ensuremath{#1}\text{--}\ensuremath{#2}~path}

\newcommand{\pths}[2]{\ensuremath{#1}\text{--}\ensuremath{#2}~paths}
\newcommand{\arc}[2]{\ensuremath{#1}\text{--}\ensuremath{#2}~arc}

\newcommand{\seq}[1]{\ensuremath{(#1_n)_{n\in\N}}} 

\newcommand{\g}{\ensuremath{G\ }}
\newcommand{\G}{\ensuremath{G}}

\newcommand{\OP}{outerplanar}
\newcommand{\Ktt}{\ensuremath{K_{3,3}}}

\newcommand{\wqo}{well-quasi-ordered}

\newcommand{\Lr}[1]{Lemma~\ref{#1}}

\newcommand{\Tr}[1]{Theorem~\ref{#1}}
\newcommand{\Trs}[1]{Theorems~\ref{#1}}
\newcommand{\Sr}[1]{Section~\ref{#1}}
\newcommand{\Srs}[1]{Sections~\ref{#1}}
\newcommand{\Prr}[1]{Pro\-position~\ref{#1}}

\newcommand{\Cr}[1]{Corollary~\ref{#1}}
\newcommand{\Cnr}[1]{Con\-jecture~\ref{#1}}

\newcommand{\Qr}[1]{Question~\ref{#1}}

\newcommand{\lf}{locally finite}

\newcommand{\scl}{star-comb lemma}
\renewcommand{\iff}{if and only if}
\newcommand{\fe}{for every}
\newcommand{\Fe}{For every}
\newcommand{\fea}{for each}

\newcommand{\st}{such that}

\newcommand{\ti}{there is}

\newcommand{\pwd}{pairwise disjoint}

\newcommand{\obda}{without loss of generality}

\newcommand{\labtequ}[2]{
 \begin{equation} \label{#1} 	\begin{minipage}[c]{0.9\textwidth}  #2 \end{minipage} \ignorespacesafterend \end{equation} }

\newcommand{\mymargin}[1]{
 \ifnum \Debug = 1
  \marginpar{%
    \begin{minipage}{\marginparwidth}\small%
      \begin{flushleft}%
        {\color{blue}#1}%
      \end{flushleft}%
   \end{minipage}%
  }%
 \fi
}%

\newcommand{\extras}[1]{
 \ifnum \Debug = 1
\section{Extras} #1
 \fi
}%

\newcommand{\mySection}[2]{}

\newcommand{\Erd}{Erd\H{o}s}

\newcommand{\DB}{\cite{diestelBook05}}

\newcommand{\LemCombStarC}{\cite[Lemma~8.2.2]{diestelBook05}} 

\begin{document}
	\title{The excluded minors for embeddability into a   compact surface}
	
\author{Agelos Georgakopoulos\thanks{Supported by  EPSRC grants EP/V048821/1 and EP/V009044/1.}}
\affil{  {Mathematics Institute, University of Warwick}\\
  {\small CV4 7AL, UK}}

\date{\today}

\maketitle

\begin{abstract}
We determine the excluded minors characterising the class of countable graphs that embed into some compact surface.            
\end{abstract}

\medskip



\section{Introduction}

The main aim of this paper is to provide the excluded minors characterising the class of countable graphs that embed into a compact surface, whereby we put no restriction on the genus. We will prove
\begin{theorem}\label{main thm}
A countable graph \g embeds into a compact (orientable) surface \iff\ it does not have one of the 8 graphs of \fig{figExSig} as a minor.\footnote{Every non-trivial infinite graph has many `minor-twins'; for example, for each pair $x,y$ of vertices of infinite degree in $\Sig_i, i\geq 5$, we could add or remove the $x$--$y$~edge.}
\end{theorem}
It is an exercise to show that none of these graphs is a minor of another. Since none of these graphs embeds into a closed surface, orientable or not, our theorem remains valid if we remove the word `orientable'. 

All graphs in this paper are countable. In the \lf\ case, only the first two obstructions $\Sig_1=\omdot K_5, \Sig_2=\omdot \Ktt$ are needed (see \Cr{cor lf}).

\begin{figure}[H]
\begin{center}
\begin{overpic}[width=.9\linewidth]{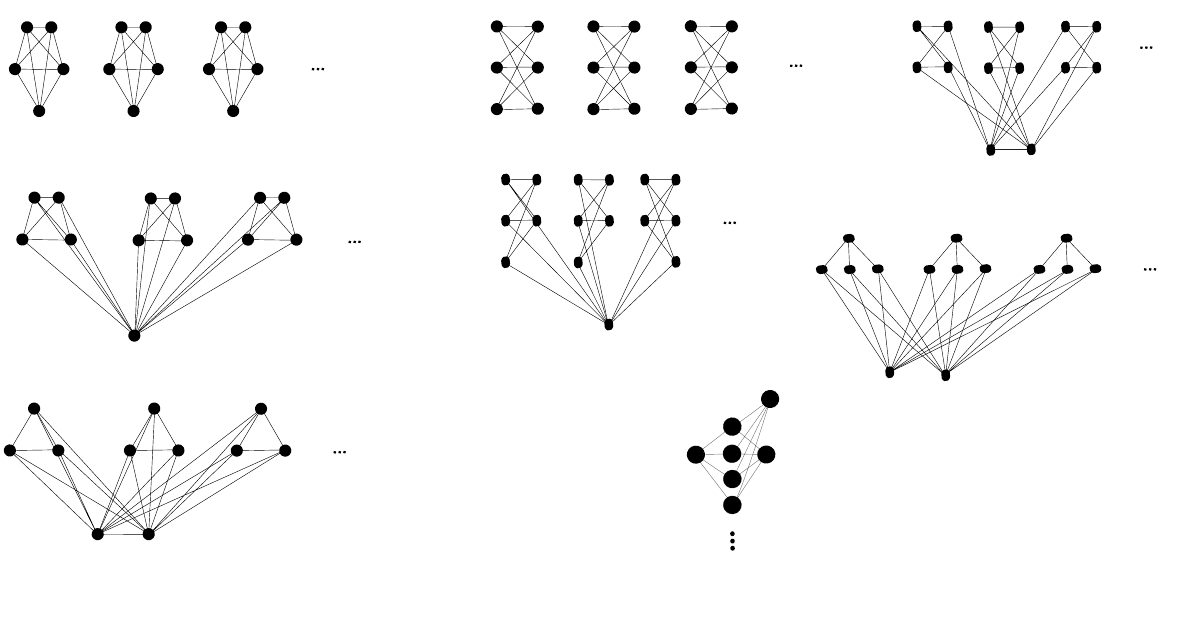} 
\put(-4,45){$\Sig_1$}
\put(37,45){$\Sig_2$}
\put(-4,31){$\Sig_3$}
\put(37,31){$\Sig_4$}
\put(73,45){$\Sig_6$}
\put(65,26){$\Sig_7$}
\put(-4,13){$\Sig_5$}
\put(55,9){$\Sig_8$}
\end{overpic}
\end{center}
\caption{The excluded minors of \Sig\ as provided by \Tr{main thm}. \newline 
    \hspace{\linewidth} $\bf \Sig_1$ (respectively $\bf \Sig_2$): the disjoint union of infinitely many copies of $K_5$ (resp.\ \Ktt);  \newline 
    \hspace{\linewidth} $\bf \Sig_3$ (resp.\ $\bf \Sig_4$):  the graph arising from $\Sig_1$ (resp.\ $\Sig_2$) by picking  one vertex from each component and identifying them; \newline 
    \hspace{\linewidth} $\bf \Sig_5$ (resp.\ $\bf \Sig_6$): the graph arising from $\Sig_1$ (resp.\ $\Sig_2$) by picking  one edge from each component and identifying them; \newline 
    \hspace{\linewidth} $\bf \Sig_7$: the graph arising from $\Sig_2$ by picking  one pair of non-adjacent vertices from each component and identifying these pairs; and 
${\bf \Sig_8 }= K_{3,\omega}$.
} \label{figExSig}
\end{figure}

An analogous statement for embeddings into a fixed surface 
is the following theorem of Robertson \& Seymour \cite{RobSeyKur} (previously announced in \cite{BKMM,christian_embedding_2015,FulKynGen}).  \Fe\ \nin, \ti\ $g$, \st\ for every graph \g of genus at least $g$, there is some $\Sig_i$ as in  \fig{figExSig}, \st\  any subgraph of size $n$ of $\Sig_i$ is a minor of \G. I do not see a way to deduce this from \Tr{main thm} or vice-versa. An important difference between these two results is that if we restrict to the orientable case, then we need to allow the $n\times n$ projective grid as a further obstruction to embeddability into a fixed surface, but \Tr{main thm} proves that we have the same excluded minors with or without the orientability restriction when allowing arbitrarily high genus.

\medskip
One of the tools for our proof of \Tr{main thm} is the following result of independent interest, saying that a graph embeds into a compact surface \iff\ it can be decomposed into finitely many planar subgraphs with finite pairwise intersections. 
\begin{theorem}\label{thm decomp Intro}
A countable graph embeds into a compact (orientable) surface \iff\ it admits a finitary decomposition into planar pieces.
\end{theorem}
See \Sr{sec decomp} for the precise definitions. \Tr{thm decomp Intro} supports the meta-conjecture that any result proved for infinite planar graphs generalises, rather easily, to graphs embeddable into a compact surface. Examples of such results include the main results of \cite{intersection,planarPB,UKtrans,HuNaUni,kozPPP}. See \Sr{sec impl} for more.

\medskip
Part of the motivation for \Tr{main thm} comes from a well-known conjecture of Thomas \cite{ThoWel} postulating that the countable graphs are well-quasi-ordered under the minor relation. The analogous statement for finite graphs is the celebrated Graph Minor Theorem of Robertson \& Seymour \cite{GMXX}. A positive answer to Thomas's conjecture would imply that every minor-closed class $\cc$ of countable graphs is characterised by forbidding a finite list $\ex{\cc}$ of \defi{excluded minors}. This is in general hard to show even for a concrete class of graphs like the class $\Sigma$ of \Tr{main thm}; indeed, it is a-priori not even clear that $\ex{\Sigma}$ is finite. Apart from the fact that $\Sigma$ is a natural class to consider, another reason why the finiteness of $\ex{\Sigma}$ is a pressing question if one is interested in Thomas's conjecture is the important role played by classes of finite graphs embeddable in a fixed surface in the proof of the Graph Minor Theorem.

Many natural minor-closed graph classes  \cc, e.g.\ the graphs embeddable into a fixed surface, have the property that a graph is in \cc\ as soon as every finite subgraph is. This has the consequence that  $\ex{\cc}$ coincides with the list of excluded minors of the subclass of \cc\ comprising its finite elements. Apart from such classes, there are very few classes \cc\ of infinite graphs for which $\ex{\cc}$ is explicitly known. The only example I am aware of are the graphs with accumulation-free embeddings in the plane \cite{HalSom}. 


Additional motivation for \Tr{main thm} comes from a question raised by Christian, Richter \& Salazar \cite{christian_embedding_2015}, asking for a characterisation of the Peano continua that embed into a closed surface analogous to Claytor's \cite{ClaPea} characterisation of the Peano continua embeddable into $\BS^2$. The special case of graph-like continua was handled in \cite{christian_embedding_2015}, and the characterisation obtained is similar to \Tr{main thm}. But the lack of compactness does not allow using that  characterisation to deduce \Tr{main thm}. Using \Tr{main thm} and the \scl\ (\Lr{SC lem} below) it is not difficult to determine the excluded \emph{topological} minors for embeddability into a compact surface, and this could be a first step towards answering the aforementioned question of Christian et al.\ \cite{christian_embedding_2015}. 

\medskip
Our proof of \Tr{main thm} is elementary (but involved), relying only on Kuratowski's theorem, and a classical result of Youngs about cellular embeddings of finite graphs. It is carried out mostly in \Srs{sec decomp}, \ref{sec UOP} and \ref{sec proof}. On the way to \Tr{main thm} we will develop techniques that allow us to find the excluded minors of families of infinite graphs that satisfy a property up to finitely many flaws: we will characterize the graphs that become forests after deleting, or contracting, finitely many edges  (\Srs{sec AF} and~\ref{sec AFR}), as well as the graphs that are \OP\ up to deleting finitely many edges (\Sr{sec OPE}).

\medskip
The \scl\ is one of the most useful tools in infinite graph theory. In \Sr{sec SC} we obtain the following strengthening for 2-connected graphs:

\begin{theorem} \label{SC 2con intro}
Let \g be a countable, 2-connected, graph, and $U\subseteq V(G)$ infinite. Then \g contains a subdivision of an infinite ladder, or of an infinite fan, or of $K_{2,\infty}$, having infinitely many vertices in $U$.
\end{theorem}
If \g is \lf, then this results in a ray in \g containing an infinite subset of $U$. 

\medskip
The aforementioned conjecture of Thomas was studied by Robertson, Seymour \& Thomas \cite{RoSeThExcI,RoSeThExcII}, and they concluded that {\it there is not much chance of proving} it, as it would have implications about the ordering of finite graphs. It is therefore natural to try to extend the Graph Minor Theorem to an intermediate level covering all finite graphs but not necessarily all countable ones. A concrete approach for doing so is offered by \Cnr{Con good} and other questions in \Sr{final} arising from our results and methods.

\section{Preliminaries} \label{prels}


We follow the terminology of Diestel \DB. We use $V(G)$ to denote the set of vertices, and $E(G)$ the set of edges of a graph \G. For $S\subseteq V(G)$, the subgraph $G[S]$ of \g \defi{induced} by $S$ has vertex set $S$ and contains all edges of \g with both end-vertices in $S$.

The \defi{degree} $d(v)=d_G(v)$ of a vertex $v$ in a graph \G, is the number of edges of \g incident with $v$. 

A \defi{ray} is a one-way infinite path. We say that \g is \defi{\lf}, if no vertex of \g lies in infinitely many edges.

\medskip

Let $G,H$ be graphs. An $H$ \defi{minor} of $G$ is a collection of disjoint connected subgraphs $B_v, v\in V(H)$ of $G$, called \defi{branch sets}, 
and edges $E_{uv}, uv\in E(H)$ of \g 
such that each $E_{uv}$ has one end-vertex in $B_u$ and one in $B_v$. We write $H<G$ to express that $G$ has an $H$ minor.



Given a set $X$ of graphs, we write $\forb{X}$ for the class of graphs $H$ \st\ no element of $X$ is a minor of $H$. 

A \defi{subdivision} of a graph \g is a graph obtained by replacing some of the edges of \g by paths with the same end-vertices. 

\medskip

A \defi{surface} is a connected 2-manifold without boundary.  An \defi{embedding}  of a countable graph \g\ into a surface $S$ is a map $f: G \to S$ from the 1-complex obtained from \g when identifying each edge with the interval $[0,1]$ to  $S$  \st\ the restriction of $f$ to each finite subgraph of $G$ is an embedding in the topological sense, i.e.\ a homeomorphism onto its image. (The reason why we restrict to finite subgraphs here is that the 1-complex topology of \g is not metrizable when \g is not \lf, and so such \g cannot have an embedding into a metrizable space $S$. For example, a star with infinitely many leaves admits an embedding into $\R^2$ in our sense but it does not admit a topological embedding. 
Let $\gamma(G)$ denote the minimum genus of an orientable surface into which a graph \g embeds.

The following is perhaps folklore, but we sketch a proof for completeness. The locally finite case has been proved by Mohar \cite[\S 5]{mohar88}. 
\begin{lemma} \label{lem emb cof}
Let \g be a  countable graph, and $S$ an orientable  surface. Then \g admits an embedding into $S$ if each of its finite subgraphs does.
\end{lemma}

When $S$ is the sphere, Dirac \& Shuster \cite{DiracSchuster} provide a proof  by an elementary compactness argument (which they atribute to \Erd).
Our proof is a combination of this with Youngs' \Tr{Youngs} below.

\begin{proof}
If $S$ has infinite orientable genus $\gamma(S)$, then it is easy to embed any  countable graph in it, so let us assume $\gamma(S)$ is finite. We may assume that \g is connected, for otherwise we can apply the result to each component of \G, using the well-known fact that the genus of a finite (disconnected) graph equals the sum of the genuses of its components \cite[Corollary 2]{BHKY}.
Suppose every finite subgraph of \g embeds into $S$. Let $\gamma:= \max_{H  \text{ is a finite subgraph of } G} \gamma(H)$. We may assume \obda\ that  $\gamma= \gamma(S)$, because if \g embeds into a surface $S'$ with $\gamma(S')< \gamma(S)$, then \g also embeds into $S$ by the classification of closed surfaces.

Thus we can pick a finite subgraph $H$ of  \G\ with $\gamma(H) = \gamma(S)$. We may assume that $H$ is connected since \g is, because adding vertices and edges to $H$ cannot decrease its genus. Let $H= G_1 \subset G_2 \subset \ldots$ be a sequence of finite subgraphs of \G, \st\ $\bigcup G_i = G$. Let $g_i: G_i \to S$ be an embedding. By a standard compactness argument (see e.g.\ \cite[\S 5]{mohar88}), there is a subsequence $\{g_{a_i}\}_{i\in \N}$ along which the restriction of $g_{a_i}$ to $H$ coincides, up to an automorphism of $S$, with a fixed embedding $g: H \to S$. By Youngs' \Tr{Youngs}, each face of $g$ is homeomorphic to an open disc. Let us first assume for simplicity that the closure $\cls{F}$ of each face $F$ of $g$ is homeomorphic to a closed disc, and then make the necessary modifications to our arguments. For each such $F$, note that the subgraph $F_i$ of $G_{a_i}$ that $g_{a_i}$ maps to $\cls{F}$ is planar. Let $G_F:= \bigcup_{i\in \N} F_i$, and note that $G_F$ is planar by the aforementioned result of Dirac \& Shuster \cite{DiracSchuster}. Even more, it follows from the arguments of \cite{DiracSchuster} that $G_F$ admits an embedding $g_F$ into $\R^2$ the outer face of which coincides with the boundary of $F$ in $g$. Thus by combining these embeddings $g_F$ with our embedding $g$ of $H$ we obtain an embedding of \g into $S$.

It remains to remove our assumption that $\cls{F}$ is homeomorphic to a closed disc for each face $F$ of $g$, but this is not difficult. If this is not the case, then the closed walk $W_F$ of $H$ bounding $F$ is not a cycle, but traverses some vertices at least twice, and this may prevent $G_F$, defined as above, from being planar. In this case, we modify $G_F$ into a planar auxiliary graph $G'_F$ as follows. For each vertex $v$ of $H$ that $W_F$ visits $i>1$ times, we introduce `copies' $v_1, \ldots, v_i$ of $v$, and replace $W_F$ by a cycle visiting each $v_j, j\leq i$ exactly once. Using the result of \cite{DiracSchuster}  as above it follows that the resulting graph $G'_F$ is planar if we distribute the edges of each $v$ to the $v_j$ appropriately; we leave some straightforward details to the reader. We can then identify all $v_j$ to a single vertex to deduce that $G_F$ embeds into $\cls{F}$, and as above combine those embeddings with $g$ to obtain an embedding of \g into $S$.
\end{proof}

We let \defi{\Sig} denote the class of countable graphs that embed into a compact orientable surface. (We will prove that every graph embeddable into a compact non-orientable surface also embeds into a compact orientable one.)

\medskip
We let \defi{$\omega$} denote the smallest infinite ordinal.
\defi{}

\subsection{The \scl} \label{sec scl}
Given a graph \G, and an infinite set $U\subseteq V(G)$, we define a \defi{$U$-star} to be a subdivision of the infinite star $K_{1,\omega}$ in \g all leaves of which lie in $U$. We define a \defi{$U$-comb} to be the union of a ray $R$ of \g with infinitely many pairwise disjoint, possibly trivial, \pths{U}{R}. We call $R$ the \defi{spine} of $C$, and the \pths{U}{R}\ its \defi{teeth}.

\begin{lemma}[Star-comb lemma {\LemCombStarC}] \label{SC lem}
Let $U$ be an infinite set of vertices in a connected graph $G$. Then $G$ contains either a $U$-star or a $U$-comb.
\end{lemma}

\section{Decomposing into planar graphs} \label{sec decomp}

The aim of this section is to prove the orientable case of \Tr{thm decomp Intro}, which will be used as a tool for the proof of \Tr{main thm}. (The non-orientable case will follow after we have proved \Tr{main thm}.)

\begin{definition} \label{def fin dec}
A \defi{decomposition} of a graph \g is a family $(G_i)_{i\in \ci}$ of subgraphs of \G, called the \defi{pieces}, \st\ $G= \bigcup_{i\in \ci} G_i$. We say that a decomposition $(G_i)_{i\in \ci}$ is \defi{finitary}, if $\ci$ is finite, and the intersection of any two distinct pieces is finite. Note that this means that at most finitely many vertices of \g lie in more than one piece.
\end{definition}

Let us collect a few lemmas for the proof of \Tr{thm decomp Intro}, and for later use. 

\begin{lemma} \label{lem S}
Let  $G \in \Sig$ and suppose $G'$ is obtained from \g by identifying two vertices $v,w\in V(G)$. Then $G' \in \Sig$.
\end{lemma}
\begin{proof}
Let $\Gamma$ be a compact orientable surface into which \g embeds. If \g is finite, then it is easy to embed $G'$ into a surface $\Gamma'$ obtained from $\Gamma$ by adding a handle. By combining this observation with \Lr{lem emb cof} we deduce that $G'$ is embeddable into $\Gamma'$, and hence $G' \in \Sig$.
\end{proof}

The power of \Lr{lem S} lies in our ability to apply it repeatedly. This way we obtain
\begin{corollary} \label{cor S plus}
Let \g be a countable graph admitting a finitary decomposition $G_1,\ldots, G_k$. If each $G_i$ lies in \Sig, then so does \G. 
\end{corollary}
\begin{proof}
Starting from the disjoint union of copies of each $G_i$, we can repeatedly apply \Lr{lem S} to identify pairs of vertices corresponding to the same vertex of $G_i$, remaining in \Sig\ after each step. Since the decomposition is finitary, we end up with $G$ after finitely many such identifications.
\end{proof}

Our last lemma is a classical result of Youngs about cellular embeddings. A \defi{face} of an embedding $g: G \to \Gamma$ of a graph into a surface is a component of $\Gamma \sm g(G)$.

\begin{theorem}[\cite{Youngs}] \label{Youngs}
Let $\Gamma$ be a closed orientable surface, and let $G$ be a finite connected graph that embeds into $\Gamma$ but does not embed into a closed orientable surface of smaller genus. Then for every embedding $g: G \to \Gamma$, each face of $g$ is homeomorphic to an open disc. 
\end{theorem}

We are now ready for the proof of the main result of this section, which we restate for convenience:
\begin{theorem}\label{thm decomp}
A countable graph \g embeds into a compact, orientable, surface \iff\ it admits a finitary decomposition into planar pieces.
\end{theorem}
\begin{proof}
The backward direction follows immediately from \Cr{cor S plus}. 

For the forward direction, it suffices to consider the case where \g is connected: at most finitely many of the components of \g can have positive genus by additivity of the genus \cite{BHKY}, and so we can work with each such component separately. 

Let $H$ be a finite subgraph of $G$ with $\gamma(H)=\gamma(G)$. Such an $H$ exists by \Lr{lem emb cof}. We may assume that $H$ is connected since \g is. 

Let $g: G \to \Gamma$ be an embedding into the closed orientable surface of genus $\gamma(G)$, and recall that we may assume $g$ to be {generous}. Note that $g$ induces an embedding $g_H: H \to \Gamma$ by restriction. By Youngs' \Tr{Youngs}, every face  $F$ of $g_H$ is homeomorphic to an open disc. Therefore, the subgraph $G_F:= g^{-1}(F)$ of $G$ embedded into $F$ is planar. Let $\cls{G_F}$ denote the subgraph of \g induced by $G_F$ and all its neighbours, which must lie on $\partial F \subset H$. If each $\cls{G_F}$ was planar, which would be the case if the closure of each $F$ was homeomorphic to a closed disc, then we could decompose \g into the $G_F$'s and be done. But the existence of such an embedding $g$ is Jaeger's strong embedding conjecture \cite{jaeger}, which is open. Instead, we will find a finitary decomposition of each $G_F$ into planar subgraphs as follows.

Let $(D_v)_{v\in V(H)}$ be open topological discs witnessing the generosity of $g$. We can assume that these discs are pairwise disjoint, by choosing sub-discs if needed. Note that $g(H)$ separates each $D_v$ into a finite number of components, each itself an open topological disc, which components we will call $D_v$-sectors. Moreover, $D_v \cap F$ is the union of some of these $D_v$-sectors \fe\ $v\in V(H)$ and every face $F$ of $H$. Define a plane supergraph $G'_F$ of $G_F$ as follows. For each $v\in V(H) \cap \partial F$, and each $D_v$-sector $s\subset F$, we introduce a new vertex $v_s$ of $G'_F$, embed $v_s$ into $s$, and reroute the edges incident with $s$ so that they attach to $v_s$ instead of $v$. 
We call these new vertices $v_s$ the \defi{boundary vertices} of $G'_F$, and denote their set by $\cb_F$. 

Note that $G'_F$ is planar by construction. It is easy to see that 
\labtequ{identify}{$G'_F$ remains planar if we identify any one pair of boundary vertices.}
Indeed, all the boundary vertices of $G'_F$ lie on its outer face, and identifying a pair of vertices on the outer face of any plane graph results in a planar graph.

By definition, $\cls{G_F}$ can be obtained from $G'_F$ by identifying each set of boundary vertices corresponding to $D_v$-sectors of the same $v\in V(H) \cap \partial F$ to one vertex. As \eqref{identify} applies to one pair only, it is not enough to guarantee that $\cls{G_F}$ is planar. Therefore, we decompose it further as follows.

For every face $F$ of $g$ as above, we pick a minimal forest $T_F$ in  $G'_F$ \st\ any two boundary vertices  that lie in the same component of $G'_F$ also lie in the same component of $T_F$. Since $\cb_F$ is finite, the existence of $T_F$ is easily guaranteed. Recall that each boundary vertex $v_s$ corresponds to some vertex $v$ of $H$; we obtain $T'_F$ by replacing each $v_s\in \cb_F$ by its corresponding $v$.  Let $H':= H \cup \bigcup_F T'_F$. Note that $T'_F$ is a subgraph of \G, hence so is $H'$. Moreover, $H'$ is connected since $H$ is. Easily, any face $F'$ of $H'$ is contained in a face $F$ of $H$. We repeat the above constructions to define $G'_{F'}$ and $\cls{G_{F'}}$ in analogy with $G'_{F}$ and $\cls{G_{F}}$. 

We claim that each component $C$ of $G'_{F'}$ contains at most two vertices of $\cb_F$. Indeed, any triple of vertices of $\cb_F \cap C$ would be contained in a subtree of $T_F$, which subtree would separate $F$ into three regions, none of which can contain the whole triple. 

If no pair of vertices in $G'_{F'} \cap \cb_F$ corresponds to the same $v\in  V(H)$, then $G'_{F'}$ is a planar subgraph of \g (and $\cls{G_F}$).  In this case we just accept  $G^*_{F'}:= G'_{F'}$ as a piece of our decomposition. If there is such a pair $v_{s_1},v_{s_2}$, then by the previous remark this pair is unique for each component $C$ of $G'_{F'}$. By identifying each such pair into a vertex we thus obtain a subgraph $G^*_{F'}$ of \G, which by \eqref{identify} is planar. Here we used the fact that a graph is planar if each of its components is. 

Note that the union of all the $G^*_{F'}$, where $F'$ ranges over all faces of $H'$, contains all edges in $E(G) \sm E(H')$. Since $H'$ is finite, its edge-set forms a finitary decomposition of $H'$. Thus we obtain the desired decomposition of \g as  the set of graphs  $G^*_{F'}$ united with the set of edges of $H'$. This is a finitary decomposition, since the intersection of any two of its elements is contained in $H'$. 
\end{proof}

\subsection{Implications of \Tr{thm decomp}} \label{sec impl}

We remark that \Tr{thm decomp} allows us to extend many results obtained for planar graphs, e.g.\ those of \cite{intersection,planarPB,UKtrans}, to graphs in \Sig. Motivated by the fact that some such results (e.g.\ \cite{HuNaUni,kozPPP}) only apply to graphs with vertex-accumulation-free embeddings into the plane, we will now formulate and prove a refinement of \Tr{thm decomp} that takes accumulation points into account. This refinement is not needed for the proof of \Tr{main thm}, and the reader may skip the rest of this section.
\medskip

We let \defi{$\Sig^*$} denote the class of countable graphs \g that embed into a compact orientable surface $\Gamma$ so that there are at most finitely points of $\Gamma$ that are accumulation points of vertices of \G. We can always choose our embeddings so that no such accumulation point lies in the image of \G. We define \defi{$\pln^*$} analogously, with $\Gamma$ replaced by $\BS^2$. Finally, we  say that \g is Vertex-Accumulation-Free, or \defi{\vapf} for short, if it admits an embedding in $\R^2$ with no accumulation point of vertices. We will prove

\begin{corollary}\label{cor vapf}
A countable graph \g lies in $\Sig^*$ \iff\ it admits a finitary decomposition into \vapf\ pieces.
\end{corollary}

Our proof will be a combination of \Tr{thm decomp} with the following basic fact about \vapf\ graphs:
\begin{lemma}[{\cite[LEMMA~7.1]{thoPla}}] \label{lem vapf}
A countable graph $H$ is \vapf,  \iff\ some embedding $g: H\to \BS^2$ has the property that \fe\ cycle $C$ one of the two sides of $g(C)$ contains only finitely many vertices.
\end{lemma}

\begin{proof}[\Pof \Cr{cor vapf}]
The proof of the backward direction follows the lines of that of  \Cr{cor S plus}, the only difference being that we start by embedding each piece into $\BS^2$ with at most one accumulation point of vertices.

\medskip
For the forward direction, given $\g \in \Sig^*$, we first apply 
\Tr{thm decomp} to obtain a finitary decomposition of \g into planar pieces, and our proof guarantees that each piece lies in $\pln^*$. Thus it now suffices to prove that each $H\in \pln^*$ admits a finitary decomposition into \vapf\ pieces. This is easy using \Lr{lem vapf}: if $H$ has $k$ accumulation points in some embedding $g$, but is not itself \vapf, then some cycle $C$ decomposes it into two pieces each embedded with fewer than $k$ accumulation points, and our result follows by induction on $k$. 
\end{proof}

\section{Graphs that have a property up to finitely many flaws} \label{sec AF} 

This section introduces classes of graphs that have a property up to finitely many `flaws', and basic techniques for finding their excluded minors. This will suffice to prove the analogue of Theorem 1.1 for locally finite graphs.

Given a minor-closed family \cc\ of infinite graphs, one can define classes of graphs that are \defi{almost} in \cc\ in the following sense. 

\begin{definition} \label{def almost} Let \defi{\rmv{\cc}} (respectively, \defi{\rme{\cc}}) denote the class of graphs \G, \st\ by removing finitely many vertices (resp.\ edges) from \g we obtain a graph belonging to \cc. Similarly, we let  \defi{\rmce{\cc}} denote the graphs that belong to \cc\ after contracting finitely many edges. 

It is easy to see that $\rmce{(\rme{\cc})} = \rme{(\rmce{\cc})}$ \fe\ \cc, and we will simply write $\rmece{\cc}$ instead.
\end{definition}

The following examples show that neither of $\cc_{\mathrm{/E}},\cc_\mathrm{E}$ is contained in the other in general. 

Example 1: Let $M$ denote the graph 
consisting of a ray emanating from the centre of an infinite star $K_{1,\omega}$ (we could call $M$ the infinite mop), and let $\cc:= \forb{M}$. Then $\cc_{\mathrm{/E}}=\cc \subsetneq \cc_\mathrm{E}$, because $\cc_\mathrm{E}$ contains $M$ while $\cc$ does not. 

Example 2: It is easy to prove $\Sig= \Sig_\mathrm{E}$ similarly to \Lr{lem S}. But $\Sig \subsetneq \Sig_\mathrm{/E}$, because $K_{3,\omega} \in \Sig_\mathrm{/E}$.

\begin{definition} \label{def omdot}
We write 
\defi{$\omdot H$} for the disjoint union of countably infinitely many copies of a graph $H$. If $H$ is  vertex-transitive, we let \defi{$\bigvee H$} denote the graph obtained from $\omdot H$ by picking one vertex from each copy of $H$ and identifying them (by vertex-transitivity, it does not matter which vertices we pick).
\end{definition} 
 For example, $\bigvee K_3$ is a \defi{bouquet of triangles}, i.e.\ an infinite union of triangles having exactly one vertex in common.

\medskip

The next proposition provides the excluded minors of the class $\frsE$ of `almost forests' (the analogous result for $\frsCE$ is given in \Sr{sec AFR}).
Although it is not formally needed, we include it here as a gentle introduction to the techniques we will later use to prove our main result (\Tr{main thm}).

\begin{proposition} \label{prop FE}
Let $\frs$ denote the class of countable forests. Then $\frsE= \forb{\omdot K_3, \bigvee K_3, K_{2,\omega}}$.
\end{proposition}
\begin{proof}
Suppose $G\not\in \frsE$. Then the set \cc\ of cycles of \g is infinite. If \cc\ contains an infinite subset consisting of pairwise disjoint cycles, then we obtain a $\omdot K_3$ minor and we are done. Otherwise, we claim that there is a vertex $v$ having infinitely many incident edges in $\bigcup \cc$. In particular, $v$ is contained in each cycle of an infinite subset $\cc' \subseteq \cc$. Indeed, if this fails for every $v\in V(G)$, then we can greedily find an infinite sequence \seq{C} of pairwise vertex-disjoint cycles, by removing, at each step $n\in \N$ all the edges of all the cycles intersecting $\bigcup_{j<n} C_j$; since the edges we thereby remove are finitely many, and  $G\not\in \frsE$, we can always find a new cycle $C_n$ avoiding $\bigcup_{j<n} C_j$.

If $v$ is the only vertex incident with infinitely many cycles in $\cc'$, then by a similar greedy construction we obtain a $\bigvee K_3$ minor in $\bigcup \cc'$, with $v$ being the vertex of infinite degree. Otherwise, let $w\neq v$ be a vertex contained in each cycle of an infinite subset $\cc'' \subseteq \cc'$. Note that $G':= \bigcup \cc'' - v$ is a connected graph, because $C - v$ contains $w$ \fe\ $C\in \cc''$. We apply the star-comb \Lr{SC lem} to $G'$ with $U$ being the set of neighbours of $v$, to obtain a $U$-star or a $U$-comb $X$. If $X$ is a $U$-star, then attaching $v$ to $X$ we obtain a subdivision of $K_{2,\omega}$ in \G. If $X$ is a $U$-comb, then contracting its ray, and attaching $v$, we again obtain a $K_{2,\omega}$ minor in \G. 

In all cases we have obtained one of $\omdot K_3, \bigvee K_3, K_{2,\omega}$ as a minor of \G.
\end{proof}

The class $\frsV$ is easier to characterise in terms of excluded minors: we have $\frsV= \forb{\omdot K_3}$. This is a special case of the following helpful fact. Its main argument is well-known in the context of Andreae's ubiquity conjecture \cite{AndUb}.

\begin{proposition} \label{prop PV}
Let $\cp= \forb{H_1, \ldots, H_k}$ be a minor-closed class of countable graphs, where the $H_i$ are finite. Then $\cp_V= \forb{\omdot H_1, \ldots, \omdot H_k}$.
\end{proposition}
\begin{proof}
If $G \not\in \cp_V$, then $G$ has an $H_i$ minor $M=M_0$ for some $i$. We may assume that each branch set of $M$ is finite since $H_i$ is finite. Notice that $G_1:= G - M$ does not belong to $\cp_V$ either. By repeating the argument to $G_1$, and continuing recursively, we obtain a sequence \seq{M} of minors of \G, with \pwd\  branch sets, each being an  $H_i$ minor. By passing to a subsequence we may assume that $i$ is fixed. Thus we have found $\omdot H_i$ as a minor of \G, establishing $\cp_V \supseteq \forb{\omdot H_1, \ldots, \omdot H_k}$.

The converse is straightforward: if $\omdot H_i<G$, then $G- F \not\in \cp$ for any finite $F\subset V(G)$ as $H_i< G-F$ holds.
\end{proof}

As an example application of \Prr{prop PV}, we deduce that $\plV=\forb{\omdot K_5, \omdot \Ktt}$, where we write \defi{$\mathrm{Planar}$} for the class of planar graphs. 
Since $\plV \subset \Sig_V$, and  $\omdot K_5, \omdot \Ktt \not\in \Sig_V$ (\cite{BHKY}; see also the proof of \Tr{main thm}) this yields

\begin{corollary} \label{cor sigV}
$\Sig_V=\plV=\forb{\omdot K_5, \omdot \Ktt}$.
\end{corollary}
An alternative proof of \Cr{cor sigV} can be obtained by using \Tr{thm decomp}: the latter implies that $\Sig \subset \plV$, by removing the intersections of the pieces of any finitary planar decomposition of $G\in \Sig$. This in turn implies $\Sig_V \subset (\plV)_V = \plV$, and so  $\Sig_V=\plV$ as the converse inclusion is trivial.

\medskip
As another corollary of \Prr{prop PV}, we obtain an easy proof of the analogue of \Tr{main thm} for \lf\ graphs:
\begin{corollary} \label{cor lf}
Let \g be a locally finite graph. Then $G\in \Sig$ unless 
$\omdot K_5 < G$ or $\omdot \Ktt<G$.
\end{corollary}
\begin{proof}
As in the proof of \Prr{prop PV}, we can find either an $\omdot K_5$ or an $\omdot \Ktt$ minor in \G, or a finite set of vertices $F\subset V(G)$ \st\ $G-F\in \Sig$. But in the latter case we easily deduce $G\in \Sig$, since $F$ is incident with finitely many edges (we could for example apply \Cr{cor S plus}).
\end{proof}

\begin{remark}
\textup{The aforementioned fact that $\frsV= \forb{\omdot K_3}$ can be thought of as the infinite version of the classical result of \Erd\ \& P\'osa \cite{ErdPosInd} saying that every finite graph has either a $k\cdot K_3$ minor or a set of at most $f(k)$ vertices the removal of which results into a forest. I do not expect there to be an easy way to deduce the one from the other, as this \Erd-P\'osa property fails for non-planar graphs in the finite case \cite{GMV}, while the infinite version holds for every finite graph by \Prr{prop PV}}.
\end{remark}

\subsection{Almost planar graphs} \label{sec AP}

The following is another consequence of \Tr{thm decomp} of independent interest. It is not needed for the proof of our other results, and the reader may skip to the next section. 
\begin{proposition} \label{prop AP}
$\rmce{\Sig} = \rmce{\pln} =\rmece{\pln}$.
\end{proposition}
\begin{proof}
The following inclusions follow immediately from the definitions:
$$\rmce{\pln} \subseteq \rmece{\pln} \subseteq \rmece{\Sig} = \rmce{\Sig},$$
where for the last equality we used the fact that $\rme{\Sig} = {\Sig}$. Thus it only remains to show that $\rmce{\Sig} \subseteq \rmce{\pln}$. For this, suppose $G \in \rmce{\Sig}$, and let $F\subset E(G)$ be finite and \st\ $G':= G / F\in \Sig$. Note that at most finitely many of the components of $G'$ are non-planar \cite{MilAdd}. For each such component $C$, apply \Tr{thm decomp} to obtain a finitary decomposition $G_1,\ldots, G_k$ of $C$ into planar pieces. Let $S$ denote the set of vertices of $C$ that lie in at least two of the pieces, and let $T$ be a finite subtree of $C$ containing $S$, which exists since $C$ is connected and $S$ finite. We claim that $C / E(T)$ is planar. Indeed, each $G_i /(E(T) \cap E(G_i))$ is planar since $G_i$ is, and $G_i \cap G_j$ contains at most one vertex after contracting $E(T)$. This proves that $C\in \rmce{\pln}$ for each non-planar component of $G'$, and therefore $G'\in \rmce{\pln}$. It follows that $G$ too lies in $\rmce{\pln}$.
\end{proof}

\section{Outerplanar graphs and related classes} \label{sec UOP}

By \Prr{prop PV} and \Cr{cor sigV}, if a graph $G$ does not lie in $\rmv{\Sig}=\rmv{\rm{Planar}}$ then it has one of the desired minors, and so the most challenging part of our proof of \Tr{main thm} is to handle the case where \g becomes planar after removing a finite vertex set $W$. When $W$ is a single vertex, the latter is tantamount to saying that $G-W$ is planar but not \OP\ \defi{relative to} the neighbourhood of $W$. The aim of this section is to characterise graphs that are \OP\ relative to a vertex set $U$, and the analogous class for embeddings in any compact surface, in terms of forbidden minors that are \defi{marked} by $U$. The precise definitions follow.


\medskip
Let \g be a graph and $U\subset V(G)$. Define the \defi{$U$-cone} $C_U(G)$ of \g to be the graph obtained from \g by adding a new vertex $u$, the \defi{cone vertex}, and joining it to each vertex in $U$ with an edge. We say that $G$ is \defi{\uop}, if $C_U(G)$ is planar. If \g is finite, it is easy to see that \g is \uop\ \iff\ it admits an embedding into $\BS^2$ \st\ all vertices in $U$ lie on a common face-boundary. Moreover, by letting $U=V(G)$ we recover the standard notion of outerplanarity: \g is outerplanar \iff\ it is $V(G)$-outerplanar.

A \defi{marked graph} is a pair consisting of a graph \g and a subset $U$ of $V(G)$, called the \defi{marked vertices}. Given two marked graphs $(G,U), (H,U')$, an $H$ \defi{marked minor} of $G$ is defined just like an $H$ minor of $G$  (see \Sr{prels}), except that for each marked vertex $v$ of $H$, we require that the corresponding branch set $B_v$ contains at least one marked vertex of $G$. We write $(G,U) <(H,U')$ when this is possible.

Our next lemma adapts the well-known fact that the finite \OP\ graphs coincide with $\forb{K_4,K_{2,3}}$. 

\begin{lemma} \label{lem Oi}
Suppose \g is a countable planar graph, and let $U\subset V(G)$. Then \g is \uop\ \iff\ $(G,U)$ does not contain one of the marked graphs $(\Theta_i, U_i), 1\leq i \leq 4$ of \fig{figOis} as a marked minor. 
\end{lemma}
\begin{figure} 
\begin{center}
\begin{overpic}[width=.7\linewidth]{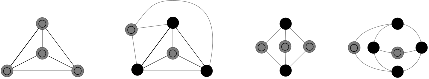} 
\put(9,-3){$\Theta_1$}
\put(39,-3){$\Theta_2$}
\put(63.5,-3){$\Theta_3$}
\put(88,-3){$\Theta_4$}
\end{overpic}
\end{center}
\caption{The excluded \mm s for relative outerplanarity. The sets of marked vertices $U_i$ are shown in grey.} \label{figOis}
\end{figure}

\begin{proof}
Since a graph is planar \iff\ each of its finite subgraphs is \cite{DiracSchuster}, we may assume for simplicity that \g is finite. 

Let $G':= C_U(G)$, as defined above. If $(\Theta_i,U_i) < (G,U)$ for some $i$, then it is easy to find a $K_5$ or $\Ktt$ minor in $G'$, hence \g  is not \uop. 

For the other direction, if $G'$ is non-planar, i.e.\ \g is not \uop, then by Kuratowski's theorem 
$G'$ contains a subdivision $K$ of $K_5$ or $K_{3,3}$. Since \g is planar, $K$ contains the cone vertex $u$. There are four simple cases depending on which of $K_5, K_{3,3}$ it coincides with, and on whether $u$ is a branch vertex or a subdivision vertex:
\begin{figure} 
\begin{center}
\begin{overpic}[width=.4\linewidth]{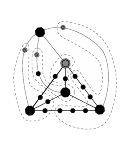} 
\put(20,83){$u$}
\end{overpic}
\end{center}
\caption{Turning $K$ into a $(\Theta_1,U_1)$ marked-minor in Case~1. The dashed enclosures represent the branch sets $B_i$, and the grey vertices represent the $v_i$.} \label{figO1}
\end{figure}
\begin{itemize}
	\item[Case 1:] \label{O i}  $K$ is a subdivision of $K_5$, and $d_K(u)=4$. In this case we obtain $(\Theta_1,U_1)$ as a \umm\ of $G$ as follows. We remove $u$ from $K$, and for each of the four neighbours $v_i\in U, 1\leq i \leq 4$ of $u$, we define a branch set $B_i$ containing the (possibly trivial) path of $K$ joining $v_i$ to a vertex $x_i\neq u$ with $d_K(x_i)=4$. We extend the $B_i$ so that each vertex of $K - u$ lies in exactly one $B_i$ (\fig{figO1}). 
	\item[Case 2:] \label{O ii} $K$ is a subdivision of $K_5$, and $d_K(u)=2$. Similarly to the previous case, we obtain $(\Theta_2,U_2)$ as a \umm\ of $G$. Indeed, $\Theta_2$ is isomorphic to $K_5$ with the edge joining the two marked vertices removed. 
	\item[Case 3:] \label{O iii} $K$ is a subdivision of $\Ktt$, and $d_K(u)=3$. Similarly to Case~1, we obtain $(\Theta_3,U_3)$ as a \umm\ of $G$. 
	\item[Case 4:] \label{O iv} $K$ is a subdivision of $\Ktt$, and $d_K(u)=2$. Similarly to Case~2, we obtain $(\Theta_4,U_4)$ as a \umm\ of $G$, whereby we use the fact that $\Theta_4$ is isomorphic to $\Ktt$ with the edge joining the two marked vertices removed.
\end{itemize}
Thus in each case we have found a $(\Theta_i,U_i)$ as a marked minor in \G.
\end{proof}

\begin{remark}
It follows from \Lr{lem Oi} that \g is \uop\ as soon as each of its finite subgraphs is.
\end{remark}


We now introduce a generalisation of \OP ity to arbitrary surfaces that will play an important role in the proof of \Tr{main thm}: 
\begin{definition} \label{def SU}
We say that a marked graph $(G,U)$ lies in \defi{\SU}, and write $(G,U)\in \SU$, if $C_U(G)\in \Sig$. 
\end{definition}

\begin{lemma} \label{lem SU}
Suppose $\g\in \Sig$ is a countable graph, and $(G,U)$ is not in \SU\ for some $U\subset V(G)$. Then $(G,U)$ contains one of the following marked graphs as a marked minor: \\
the graphs
$\Phi_i, 1\leq i \leq 5$, $\Phi'_i, 2\leq i \leq 4$ of \fig{figUis}, or \\ the graphs $\omega \cdot \Theta_i, 1\leq i \leq 4$, with $\Theta_i$ as in \fig{figOis}.
\end{lemma}
\begin{figure} 
\begin{center}
\begin{overpic}[width=1\linewidth]{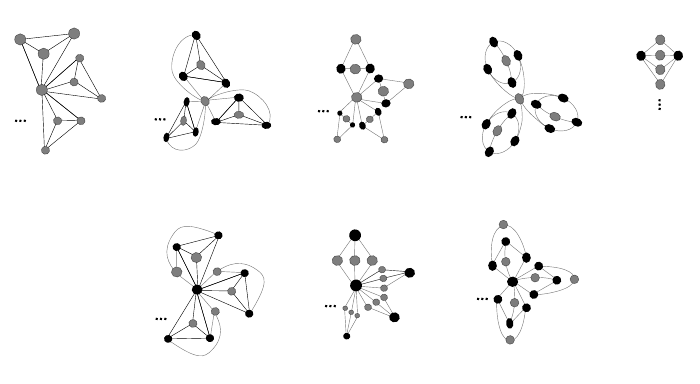} 
\put(5,29){$\Phi_1$}
\put(26,29){$\Phi_2$}
\put(49,29){$\Phi_3$}
\put(70,29){$\Phi_4$}
\put(93,29){$\Phi_5$}
\put(26,0){$\Phi'_2$}
\put(49,0){$\Phi'_3$}
\put(70,0){$\Phi'_4$}
\end{overpic}
\end{center}
\caption{Some of the excluded \mm s of \SU. The grey vertices represent the marked ones.} \label{figUis}
\end{figure}

\Lr{lem SU} is the technically most challenging part of the proof of \Tr{main thm}. We prepare its proof with a number of lemmas. The first one is similar to \Cr{cor S plus}.

\begin{lemma} \label{cor SU}
Let \g be a countable graph admitting a finitary decomposition $G_1,\ldots, G_k$. If each $(G_i,U\cap V(G_i))$ lies in \SU\ 
for some $U\subseteq V(G)$, then $G\in \SU$.
\end{lemma}
\begin{proof}
Our assumption says that $C_{U}(G_i)$ lies in $\Sig$ \fe\ $i$. (To simplify notation we write $C_{U}(G_i)$ instead of $C_{U \cap V(G_i)}(G_i)$ throughout this proof.) Let $G'$ be the graph obtained from the disjoint union of the $C_{U}(G_i), i\leq k$, by identifying each set of vertices corresponding to the same vertex of \g into one vertex. 
By applying \Cr{cor S plus} to $G'$, decomposed into the $C_{U}(G_i), i\leq k$,  we deduce that that $G' \in \Sig$. 

Note that $C_U(G)$ can be obtained from $G'$ by identifying the cone vertices of each $C_{U}(G_i)$ into one cone vertex. Thus by applying \Lr{lem S} $k-1$ times, we deduce that $C_U(G)\in \Sig$, i.e.\ $G\in \SU$.
\end{proof}

Using this, we can now extend \Lr{lem Oi} from planar graphs to graphs in \Sig:

\begin{lemma} \label{lem STh}
Let \g be a  graph in \Sig, and let $U\subset V(G)$. Then $G$ lies in  $\SU$ \iff\ it does not contain one of the marked graphs $(\Theta_i,U_i), 1\leq i \leq 4$ of \fig{figOis} as a marked minor. 
\end{lemma}
\begin{proof}
Let $G_1,\ldots, G_k$ be a finitary decomposition of \G\ into planar pieces, as provided by \Tr{thm decomp}. The backward directions is straightforward, so assume $G \not\in \SU$. If each $G_i$ is $(U\cap V(G_i))$-\OP, then \Lr{cor SU} implies that $G \in \SU$ contrary to our assumption. Thus some $G_i$ is ---planar and--- not $(U\cap V(G_i))$-\OP, and plugging it into \Lr{lem Oi} yields the desired minor.
\end{proof}

Given a marked graph $(G,U)$ with $G\in \Sig$, and a vertex $x\in V(G)$, we say that $x$ is \defi{\SU-critical}, if $(G,U)$ is not in $\SU$ but $(G-x,U-x)$ is. The most difficult part of the proof of \Lr{lem SU} lies in finding a $\Phi_5$ minor in the case where \g contains at least two \SU-critical vertices. This is achieved (and refined) by the following two lemmas. Recall the definition of a $U$-star from \Sr{sec scl}.

\begin{lemma} \label{lem SU star}
Suppose $G \in \Sig$, and $x\in V(G)$ is \SU-critical for some $U\subseteq V(G)$. Then \g contains a $U$-star with $x$ as the infinite-degree vertex.
\end{lemma}
\begin{proof}
Finding the desired $U$-star is tantamount to finding an infinite set of pairwise disjoint \pths{N(x)}{U}, where $N(x)$ stands for the set of vertices sending an edge to $x$. This is possible unless there is a finite set $S$ of vertices separating $N(x)$ from $U$ in $G-x$. Indeed, if there is a finite and maximal set of pairwise disjoint \pths{N(x)}{U}, then their union forms such an $S$. (We have just used a trivial infinite version of Menger's theorem.) Assuming such an $S$ exists, let $D_x$ denote the component of $G -S$ containing $x$, and let $\dbar{D_x}:= G[D_x \cup S]$ denote the subgraph of \g induced by $D_x \cup S$. Note that $N(x)\subseteq V(\dbar{D_x})$ by the definitions. Let $D_U$ denote the union of all other components of $G -S$, and let $\dbar{D_U}:=G[D_U \cup S]$. Note that $U\subseteq V(\dbar{D_U})$. 

Thus $\dbar{D_x},\dbar{D_U}$ is a finitary decomposition of \G, as $V(\dbar{D_x} \cap \dbar{D_U}) \subset S$. We will use \Lr{cor SU} to deduce that $G\in \SU$, contradicting our assumptions. Indeed, both $\dbar{D_x}, \dbar{D_U}$ lie in \Sig\ being subgraphs of \G. Since $\dbar{D_x}$ contains at most finitely many elements of $U$ (those in $S\cap U$), it lies in \SU. Since $\dbar{D_U}$ is a subgraph of \g that avoids the \SU-critical vertex $x$, we have $\dbar{D_U} \in \SU$ too. Thus $G\in \SU$ by \Lr{cor SU}. This contradicts the existence of $S$, thus proving the existence of the desired $U$-star.


\end{proof}

We use \Lr{lem SU star} in order to prove 

\begin{lemma} \label{lem SU dstar}
Suppose $G \in \Sig$, and \g has two \SU-critical vertices $x,y$ for some $U\subseteq V(G)$. Then \g contains  the marked double-star $\Phi_5$ as a \mm. 
\end{lemma}
\begin{proof}
By \Lr{lem SU star} \g contains a $U$-star $T_x$ centred at $x$. By Zorn's lemma we can choose $T_x$ to be maximal, i.e.\ \st\ there is no \pth{x}{U}\  avoiding $T_x-x$. By repeating with $x$ replaced by $y$, we obtain a maximal,  infinite,  $U$-star $T_y$ centred at $y$. 

If \g contains an infinite set \cp\ of pairwise disjoint \pths{T_x}{T_y}, e.g.\ when $T_x \cap T_y$ is infinite, then it is easy to construct the double-star as a \mm\ of $T_x \cup \bigcup \cp \cup T_y$ greedily.

Thus from now on we may assume that no such \cp\ exists, and therefore there is a finite set of vertices $S$ separating $T_x$ from $T_y$ in \g by (the trivial version of) Menger's theorem. Similarly to the proof of \Lr{lem SU star}, we let $C_x$ denote the union of the components of $G -S$ intersecting $T_x$, and let $\dbar{C_x}:= G[C_x \cup S]$. We let $\dbar{C_y}$ be the subgraph of \g induced by $S$ and all other components of $G -S$, and note that $T_y\subseteq \dbar{C_y}$. Since $G$ decomposes into  $ \dbar{C_x}, \dbar{C_y}$, and $G\not\in \SU$, \Lr{cor SU} implies that at least one of $\dbar{C_x} ,\dbar{C_y}$, say $\dbar{C_y}$, is not in \SU. Note that $x$ must lie in $\dbar{C_y}$ (hence in $S$) because $G-x\in \SU$ by assumption. Thus $x$ must be \SU-critical in $\dbar{C_y}$ since it is \SU-critical in \G. Applying \Lr{lem SU star} to $x$ in $\dbar{C_y}$, we thus obtain a $U$-star $T$ in $\dbar{C_y}$ with $x$ as the centre. Only finitely many vertices of $T$ lie in $S$, and so almost all of $T$ is disjoint from $T_x$, contradicting the maximality of the latter. This contradiction proves that \cp\ exists, hence so does the desired double-star.
\end{proof}

We can now prove the main result of this section:

\begin{proof}[Proof of \Lr{lem SU}]
We claim that 
\labtequ{crit}{\g has
either  an $\omdot \Theta_i$ \mm\ for some $1\leq i \leq 4$, or a subgraph $G'\in \Sig$ with $(G' ,U \cap V(G')) \not\in \SU$ containing a  \SU-critical vertex.}
Indeed, by \Lr{lem STh} we can find a marked $\Theta_i$ minor $\Theta$ in $G$ for some $1\leq i \leq 4$. Easily, we can choose the branch sets of $\Theta$ to be finite. We can now recursively remove the vertices in the branch sets of $\Theta$ one by one, and we will either encounter a \SU-critical vertex of $G-F$ for some finite $F\subset V(G)$, or deduce that $G-\Theta$ is still not in \SU, in which case we can apply \Lr{lem STh} again to find another $\Theta_i$ \mm\ disjoint from $\Theta$. Continuing like this ad infinitum, we achieve one of our two aims of \eqref{crit}. 

If we thereby obtain an $\omdot \Theta_i$ \mm\ our task is complete, so we can assume from now on that some subgraph $G' \subseteq G$ in $\Sig \sm \SU$ has a  \SU-critical vertex  $x$. We may assume \obda\ that $G'=G$, since both are graphs in $\Sig \sm \SU$. By modifying \eqref{crit} slightly, we will next prove 
\labtequ{crit2}{\g has
either  a $\Phi_i$ or $\Phi'_i$ \mm\ for some $1\leq i \leq 4$, or a subgraph $G'\in \Sig \sm \SU$ containing two  \SU-critical vertices.}
The proof is similar to that of \eqref{crit}: we recursively apply 
 \Lr{lem STh} to find a marked $\Theta_i$ minor $\Theta$ in $G$, and remove the vertices inside $\Theta$ other than $x$ one by one, until we encounter a second \SU-critical vertex $y\neq x$ of $G-F$ for some finite $F\subset V(G)$. If this happens, we note that $x$ remains \SU-critical in $G-F$ by the definitions, and so the second option of \eqref{crit2} is satisfied. If we never encounter such a vertex $y$, then we continue ad infinitum to find an infinite set \cc\ of $\Theta_i$ \mm s of \g intersecting at $x$ only. 
 
 We can easily turn \cc\ into a $\Phi_i$ or $\Phi'_i$ minor as follows. By the pigeonhole principle, we can find an infinite subset $\cc' \subseteq \cc$ all elements of which coincide with $\Theta_i$ for a fixed $i$, and moreover their branch set containing $x$ always corresponds to a  marked vertex of $\Theta_i$ or always corresponds to an unmarked one. We form a minor $M$ of $\bigcup \cc' \subseteq G$ as follows. For each $C\in \cc'$ we let $B_x$ denote the branch set of $C$ containing $x$, and we declare their union $Y:= \bigcup_{C\in \cc'} B_y$ to be a branch set of $M$. Every other branch set of any $C\in \cc'$ is declared to be a branch set of $M$ too. The latter are pairwise disjoint since $x$ is the only common vertex of any two elements of $\cc'$. This $M$ is a $\Phi_i$ or $\Phi'_i$ minor of \g as in our statement: if each element of $\cc'$ is a $\Theta_i$ minor, then $M$ is a $\Phi_i$ or $\Phi'_i$ minor; it is the former if the branch set containing $x$ always corresponds to a marked vertex, and the latter if it always corresponds to an unmarked one.
 
This completes the proof of \eqref{crit2}, from which our statement immediately follows: if \eqref{crit2} returns a $\Phi_i$ or $\Phi'_i$ minor then our task is complete, and if it returns a pair of \SU-critical vertices we apply  \Lr{lem SU dstar} to obtain a $\Phi_5$ minor. Thus in every possible case we have found one of the desired $U$-marked graphs as a minor of \G.
\end{proof}

\begin{remark}
\textup{The converse of \Lr{lem SU} holds too, that is, if \g has one of the \umm s as in the statement, then \g is not in \SU. Indeed, the $U$-cone of each of these graphs is not in \Sig,  because such a cone contains one of the forbidden structures of \Tr{main thm} as we will see in the proof of \Tr{main thm}.} 
\end{remark}


\section{The excluded minors of \Sig} \label{sec proof}

We can now prove our main result.
\begin{proof}[Proof of \Tr{main thm}]
Let us start with the easy direction, that if \g has a minor as in the statement, then  $G \not\in \Sig$. It suffices to show that none of these graphs lies in $\Sig$. This is indeed the case, as finite subgraphs of those graphs are well-known to have unbounded Euler genus; see e.g.\ \cite{RiSiCro} for $\Sig_8=K_{3,\omega}$, and \cite{MilAdd} for the other seven $\Sig_i$. 

In particular, none of the $\Sig_i$ embed in a closed non-orientable surface either, and so we can drop the word `orientable' in the statement of \Trs{main thm} and \ref{thm decomp Intro}. \mymargin{Mention projective grids?}


\medskip
For the other direction, suppose $G \not\in \Sig$. If $G \not\in \Sig_V$, then by  \Cr{cor sigV} we can find $\omega \cdot K_5$ or $\omega \cdot \Ktt$ as a minor of \G, and so our statement holds in this case. 

So assume  $G \in \Sig_V$, and choose a smallest set of vertices $F=\{v_1,\ldots,v_k\}$ \st\ $G-F \in \Sig$. We may assume \obda\ that $|F|=1$, i.e.\ $G -v_1\in \Sig$, because we could  replace $G$ by $G- \{v_2,\ldots,v_k\}$.

Let $U= N(v_1)$ be the set of vertices of $G':= G-v_1$ sending an edge to $v_1$. 
If $(G',U)$ lies in \SU, then \g lies in \Sig\ by the definitions, and we have a contradiction. 

Thus we can assume from now on  that $(G',U)$ does not lie in \SU, and apply \Lr{lem SU}, to obtain $(G',U)> X$ where $X$ is one of the excluded marked graphs of that Lemma. We use use the singleton $v_1$ as an additional branch set to obtain the $U$-cone $X':= C_U(X)$ as a minor of \G. Ignoring the marking, we will thus obtain one of the graphs of \fig{figExSig} as a minor, depending on which graph $X$ coincides with, as follows.

\medskip
The easiest case is where $X = \Phi_1$, in which case we have $X' = \Sig_5$.  \\
If $X = \Phi_2$, we obtain $X' > \Sig_3$ by forming a branch set comprising the two infinite-degree vertices of $X'$. \\
If $X = \Phi'_2$, we have $X' > \Sig_6$; to see this, notice that after removing the central vertex of $\Phi'_2$, each component has a 4-cycle in alternating colours. \\ 
If $X = \Phi_3$, we clearly have $X' = \Sig_6$. \\
If $X = \Phi'_3$, we clearly have $X' = \Sig_7$. \\
If $X = \Phi_4$, then similarly to the  $X = \Phi_2$ case, we obtain $X' > \Sig_4$ by forming a branch set comprising the two  infinite-degree vertices of $X'$. \\
If $X = \Phi'_4$, we have $X' > \Sig_6$; to see this, remove the edge from the central vertex $w$ of $\Phi'_4$ to each grey vertex, and contract the other edge incident with that grey vertex. Each component of $\Phi'_4-w$ thus becomes a 4-cycle with alternating colours. \\
If  $X = \Phi_5$,  we have $X' = \Sig_8 = K_{3,\omega}$. 

\medskip
If $X = \omdot \Theta_1$, we have $X' = \Sig_3 = \bigvee K_5$ (\fig{figOis}). \\
If $X$ is $\omdot \Theta_3$,  we have $X' = \Sig_4= \bigvee \Ktt$. \\
If $X =\omdot \Theta_2$, we obtain  $X' > \Sig_3$ by contracting one of the two marked vertices $u_i$ of each copy $T_i$ of $\Theta_2$ onto the cone vertex $v_1$; note that this contracted vertex inherits the three edges from $u_i$ to  $T_i$, as well as a fourth edge from $v_1$ to the remaining vertex of $T_i$. \\
Finally, if $X = \omdot \Theta_4$, then similarly we obtain  $X' > \Sig_4$ by contracting one of the two marked vertices of each copy of $\Theta_4$ onto $v_1$.

To summarize, assuming $G \not\in \Sig$, we obtained one of the graphs of \fig{figExSig} as a minor of \G.
\end{proof}

\section{Almost outerplanar graphs} \label{sec OPE}

The aim of this section is to prove the analogue of \Prr{prop FE} for \OP\ graphs. This is included as a result of independent interest, proved using some of the techniques developed above. 

\medskip
Let \defi{\OuPl} denote the class of countable \OP\ graphs.
Let $G_1$ (respectively, $G_2$) be the graph obtained from $\omdot K_{2,3}$ by choosing a vertex of degree 3 (resp.\ degree 2) from each copy of $K_{2,3}$ and identifying them (\fig{figG1G2}). 
\begin{figure} 
\begin{center}
\begin{overpic}[width=.6\linewidth]{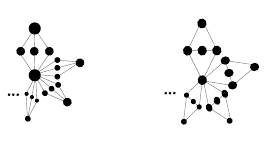} 
\put(11,0){$G_1$}
\put(72,0){$G_2$}
\end{overpic}
\end{center}
\caption{Two of the excluded minors of \Prr{prop OP}, arising by combining infinitely many copies of $K_{2,3}$.} \label{figG1G2}
\end{figure}

\begin{proposition} \label{prop OP}
$\rme{OP}= \forb{\omdot K_4, \omdot K_{2,3}, \bigvee K_4, G_1,G_2, K_{2,\omega}}$.
\end{proposition}

\begin{proof}
Suppose $G\not\in \rme{OP}$. If $G\not\in \rmv{OP}$, then by \Prr{prop PV}, and the well-known fact that $\OuPl=\forb{K_4, K_{2,3}}$, we deduce that \g contains one of $\omdot K_4, \omdot K_{2,3}$ as a minor, and we are done. Thus we may assume that $G\in \rmv{OP}$. Choose a finite, minimal $W\subset V(G)$ \st\ $G-W\in \rme{OP}$, which exists since we can even achieve $G-W\in {OP}$.  Pick $v\in W$, and let $G':= G - (W\sm \{v\})$. Then $G'\not\in \rme{OP}$, but $G'- v\in \rme{OP}$ by the choice of $W$. We will find one of the desired minors in $G'$.

Note that there must be infinitely many subdivisions of $K_4$ or $K_{2,3}$ in $G'$ containing $v$. Even more, 
\labtequ{K4s}{there is an infinite set $\cc$ of subdivisions of $K_4$ or $K_{2,3}$ in $G'$ containing $v$, no two of which share an edge of $v$.}
Indeed, if not, then there is a finite set $F$ of edges of $v$ \st\ $G'-F$ contains no subdivision of $K_4$ or $K_{2,3}$ containing $v$, which contradicts $G'\not\in \rme{OP}$.

We distinguish two cases. If $v$ is the only vertex contained in infinitely many elements of $\cc$, then we can greedily find an infinite subset $\cc' \subseteq \cc$ \st\ $C\cap D= \{v\}$ holds \fe\ $C,D\in \cc'$. We may further assume that all elements of $\cc'$ are copies of $K_4$ or they are all copies of $K_{2,3}$, and that $v$ has the same degree ---either 2 or 3--- in each of them. Thus we find one of $\bigvee K_4, G_1,G_2$ as a minor of $\bigcup \cc'$, with $v$ being the vertex of infinite degree. 

Otherwise, \ti\ a vertex  $w\neq v$ contained in each element of an infinite subset $\cc' \subseteq \cc$. Note that $G'':= \bigcup \cc' - v$ is a connected graph, containing infinitely many neighbours of $v$ by \eqref{K4s}. We apply the star-comb lemma  to $G''$ with $U$ being the set of neighbours of $v$, to obtain a $K_{2,\omega}$ minor in $\G'$ as in the proof of \Prr{prop FE}. 

In all cases we have obtained one of $\omdot K_4, \omdot K_{2,3}, \bigvee K_4, G_1,G_2 , K_{2,\omega}$ as a minor of \G.
\end{proof}

\section{Almost forests revisited} \label{sec AFR}

The aim of this section is to prove the following result, which complements \Prr{prop FE}.

\begin{proposition} \label{prop FCE}
Let $\frs$ denote the class of countable forests. Then $\frsCE= \forb{\omdot K_3, \bigvee K_3} = \rmece{\frs}$.
\end{proposition}
Combined with \Prr{prop FE}, it follows that $\frsE \subsetneq \frsCE$. For our proof we will need the following extension of the star-comb lemma.

A \defi{2-star} is a graph obtained from the star $K_{1,\omega}$ by subdividing each edge at least once. In other words, a 2-star is obtained from the disjoint union of infinitely many paths of length at least 2 by identifying their first vertices.

We say that a vertex-set $D$ \defi{dominates} another vertex-set $U$ of a graph, if $\cls{N}(D) \supseteq U$, where  $\cls{N}(D)$ consists of $D$ and all vertices sending an edge to $D$.

\begin{lemma} \label{2star comb}
Let $G$ be a connected graph, and $U\subseteq V(G)$. Then \g contains at least one of the following:
\begin{enumerate}
	\item \label{SC i} A $U$-comb;
	\item \label{SC ii} a 2-star with leaves in $U$;
	\item \label{SC iii} a finite vertex-set dominating $U$.
\end{enumerate}
\end{lemma}
\begin{proof}
We say that a vertex $v$ is \defi{$U$-dominant}, if $|N(v) \cap U|=\infty$, where $N(v)$ stands for the set of neighbours of $v$. Let $D$ denote the set of $U$-dominant vertices of \G.

If $|D|=\infty$, we apply the star-comb lemma to $G,D$, to obtain either a $D$-star $S$ or a $D$-comb $K$. In the former case, we construct a 2-star with leaves in $U$ as follows. Let $\seq{w}$ be an enumeration of the leaves of $S$, and let $s$ be the centre of $S$. Pick a neighbour $u_0\neq s\in U$ of $w_0$, and add the $w_0$--$u_0$~edge to $S$. If $u_0\in V(S)$, then we also \defi{delete} the subpath of $S - s$ containing $u_0$. For $i=1,2,\ldots$, we proceed similarly: we let $w_{a_i}$ be the next leaf of $S$ that has not been deleted, we pick a neighbour $u_i\in U$ of $w_{a_i}$ that does not coincide with $s$ or any $u_j, j<i$, attach it to $w_{a_i}$, and delete the subpath of $S - s$ containing $u_i$ if it exists. Such a $u_j$ always exists, because $w_{a_i}\in D$ has infinitely many neighbours in $U$. After $\omega$ steps we have transformed $S$ into the desired 2-star $S'$ witnessing \ref{SC ii}. 

In the other case where we obtain a comb $K$, then either $|K\cap U|=\infty$, and we are done as \ref{SC i} holds, or way may assume that $|K\cap U|=\emptyset$ by replacing $K$ with a sub-comb. In this case we construct a comb $K'$ by attaching leaves in $U$ to infinitely many vertices of $K\cap D$ by imitating the way we obtained $S'$ from $S$ (except that now we never have to delete anything).

\medskip
Thus we have settled the case $|D|=\infty$, and we now proceed to the case where $|D|<\infty$. Suppose first that $U':= U \sm \cls{N}(D)$ is finite. Then $D \cup U'$ is a finite set dominating $U$, i.e.\ option \ref{SC iii} holds. 

So suppose $U'$ is infinite. We apply the star-comb \Lr{SC lem} to \g and $U'$. If it returns a $U'$-comb then option \ref{SC i} holds. If it returns a $U'$-star $S$, then as no vertex is $U'$-dominant by the definition of $U'$, it follows that $S$ contains a 2-star with leaves in $U'$, which means that option \ref{SC ii} holds. 

\comment{ 	\begin{itemize}
	\item[Case i:] \label{D i} No component of $G-\cls{N}(D)$ has infinite intersection with $U'$. In this case we obtain a 2-star with leaves in $U$ as follows. For each of the infinitely many components $C_i$ of $G-\cls{N}(D)$ intersecting $U'$, let $P_i$ be a \pth{U'}{{N}(D)}\ in \G. If infinitely many of these $P_i$ terminate at a common vertex $w$ of ${N}(D)$, then their union contains the desired 2-star; indeed, since $w\not\in D$, at most finitely many of the $P_i$ terminating at $w$ comprise a single edge. If, on the other hand, only finitely many $P_i$ terminate at each $w\in {N}(D)$, then we can choose an infinite subset of pairwise disjoint $P_i$, extend each of them by one edge to reach $D$, and find an $s\in D$ at which infinitely many of these $P_i$ terminate. We thus obtain the desired 2-star, centred at $s$.
	\item[Case ii:] \label{D i} Some component $C$ of $G-\cls{N}(D)$ has infinite intersection with $U'$. In this case we apply the star-comb lemma to the pair $C,(U'\cap C)$. Since $C \cap D=\emptyset$, this yields that one of \ref{SC i},  \ref{SC ii} holds.
	\end{itemize}
}
Thus in every case we have obtained one of the desired sub-structures.
\end{proof}

\begin{proof}[Proof of \Prr{prop FCE}]
Suppose $G\not\in \frsCE$. 
If $G\not\in \frsV$, then by \Prr{prop PV} we deduce that \g has a $\omdot K_3$ minor, and we are done. Thus we may assume that $G\in \frsV$. Choose a finite $W\subset V(G)$ \st\ $G-W\in \frs$, minimal with this property. Assume first that \g is connected. Thus we can find a finite subtree  $T$ of \g containing $W$. Let $G':= G/T$ be the minor of \g obtained by contracting $T$ onto a vertex $v$. Then $G'\not\in \frsCE$ since $G\not\in \frsCE$ and $T$ is finite. Moreover, $F:= G'-v\in \frs$. We will prove that $G'$, and hence \G, has a $\bigvee K_3$ minor.

If there is an infinite set $\{C_n\}_{\nin}$ of components of $F$ \st\ $v$ sends at least two edges $e_n,f_n$ to $C_n$, then by appending an \pth{e_n}{f_n}\ path through $C_n$ for each $n$ we obtain a $\bigvee K_3$ minor in $\{v\} \cup \bigcup_{\nin} C_n$. Thus we may assume that there is a finite set $\{C_1, \ldots, C_k\}$ of components of $F$ sending at least two edges to $v$, and all other components of $F$ send at most one edge to $v$. Note that $G'-(C_1 \cup \ldots \cup C_k)$ is a forest.  
 
For each $i=1,\ldots, k$, we let $G'_i$ be the subgraph of $G'$ induced by $C_i \cup \{v\}$, and let $U_i$ be the set of vertices of $C_i$ sending an edge to $v$. We apply \Lr{2star comb} to $G'_i-v,U_i$, and obtain a $U_i$-comb, or a 2-star with leaves in $U_i$, or a finite set $W_i$ of vertices dominating $U_i$. If this yields a $U_i$-comb $C$, then we easily obtain a subdivision of $\bigvee K_3$ in $G'_i[C \cup \{v\}]$. If instead we obtain a 2-star $S$, with centre $s$ and leaves in $U_i$, then by contracting a \pth{v}{s}\ we obtain a subdivision of $\bigvee K_3$ in $S \cup \{v\}$, centred at the contracted path. Finally, if \Lr{2star comb} returns a finite dominating set $W_i$, then we let $T_i$ be a finite subtree of $G'_i$ containing  $W_i \cup \{v\}$. Note that $G'_i/T_i$ is a tree. 
%
%
To summarize, either we have obtained a $\bigvee K_3$ minor in \G, or a finite $T_i\subset G'_i$ \fea\ $1\leq i \leq k$ \st\ $G'_i/T_i$ is a tree. But the latter implies that contracting each $T_i$ turns $G'$ into a forest, contradicting the fact that $G'\not\in \frsCE$.

It remains to handle the case where \g is disconnected. In this case we apply the same reasoning to one of the components of \g intersecting $W$, and find a $\bigvee K_3$ minor there.

\medskip
To deduce that $\frsCE= \rmece{\frs}$, note that $\frsCE \subseteq \rmece{\frs}$, and that none of the excluded minors $\omdot K_3,  \bigvee K_3$ of the former lies in the latter. Thus $\rmece{\frs} \sm \frsCE$ is empty.
\end{proof}

\section{A star-comb lemma for 2-connected graphs} \label{sec SC}

While trying to prove \Tr{main thm} I came up with the following strengthening of the \scl\ for 2-connected graphs. Although it is not used for any of our proofs, I decided to include as it might become useful elsewhere. The \scl\ is one of the most useful tools in infinite graph theory. Some other strengthenings were obtained in a recent series of 4 papers by B\"urger \& Kurkofka \cite{BurKurDuaI}--\cite{BurKurDuaIV}. A related result determining unavoidable induced subgraphs for infinite 2-connected graphs is obtained by Allred, Ding \& Oporowski \cite{AlDiOpUna}. 

In analogy with $U$-stars and $U$-combs as in the statement of the \scl, we introduce the following structures. A \defi{double-star} is a subdivision of $K_{2,\omega}$. A \defi{ladder} consists of two disjoint rays $R,L$ and an infinite collection of pairwise disjoint \pths{R}{L}. A \defi{fan} consists of a ray $R$, a vertex $d\not\in V(R)$, and an infinite collection of \pths{d}{R}\ having only $d$ in common. \mymargin{(\fig{})} For each of these three terms, adding the prefix \defi{$U$-} means that the structure has infinitely many of its vertices in $U$. With this terminology, \Tr{SC 2con intro} from the introduction can be formulated as follows.

\begin{theorem} \label{SC 2con}
Let \g be a 2-connected graph, and $U\subseteq V(G)$ infinite. Then \g contains a $U$-double-star, or a $U$-ladder, or a $U$-fan. 
\end{theorem}


The following follows from the statement of \Tr{SC 2con}, but we need to prove it first as a first step towards the proof of the latter.
\begin{lemma} \label{comb 2con}
Let \g be a 2-connected, \lf\ graph, and $U\subseteq V(G)$ infinite. Then \g has a ray containing an infinite subset of $U$.
\end{lemma}

In this section we assume that the reader is familiar with the basics about the end-compac\-ti\-fication of a graph, and normal spanning trees; we refer to \cite{diestelBook05} therefor.
\begin{proof}
Let $\chi$ be an accumulation point of $U$ in the end-compactification of \G, and choose $U'\subseteq U$ converging to $\chi$. Let $R$ be a ray converging to $\chi$, \st\ each component of $G - R$ sends only finitely many edges to $R$; we could for example choose $R$ inside a normal spanning tree of \G; see \cite[Exercise~8.27]{diestelBook05} or \cite[Lemma~11]{fleisch}. If $R\cap U'$ is infinite then we are done, so assume this is not the case. 

Note that no component of $G - R$ can contain an infinite subset of $U'$. Thus there is an infinite set $\seq{C}$ of components of $G - R$ each intersecting $U'$. For each $C_n$, pick $u_n\in C_n \cap U'$, and two disjoint \pths{u_n}{R}. The union of these two paths is a path $P_n$ through $C_n$ with end-vertices $x_n,y_n$ on $R$. Since \g is \lf, we can easily choose an infinite subset $\seq{C'}$ \st\ the subpaths $R_n$ of $R$ from $x_n$ to $y_n$ are pairwise disjoint. By replacing each $R_n$ with $P_n$ we thus transform $R$ into a ray $R'$ containing infinitely many elements of $U'$.

\end{proof}

\begin{proof}[Proof of \Tr{SC 2con}]
Suppose first that \g is \lf. By \Lr{comb 2con}, \g has a ray $R$ with $R\cap U$ infinite. Let $\chi$ be the end of \g containing $R$. Halin \cite{halin74} proved that 
we can find two disjoint rays belonging to $\chi$.
 It is proved in \cite[Lemma~10]{fleisch} that we can choose these two rays $X,Y$ so that they intersect every ray of $\chi$ infinitely often. In particular, $R \cap (X\cup Y)$ is infinite. Since $X,Y$ belong to the same end, we can find a sequence \seq{P}\ of pairwise disjoint \pths{X}{Y}\ \st\ $\bigcup_\nin P_n \cup X\cup Y$ forms a ladder $L$. By replacing some subpaths of $L$ by subpaths of $R$ we can easily obtain a ladder $L'$ containing infinitely many elements of $U$. This settles the case where \g is \lf.
\medskip

If \g is not \lf, then we can still assume it is countable. For if not, then we can find a countable 2-connected subgraph $G'$ with $V(G) \cap U$ infinite, and find the desired structure in $G'$. Indeed, we can pick a countably infinite subset $\{u_1, u_2, \ldots \}$ of $U$, choose a pair of independent paths from each $u_i$ to the two end-vertices of a fixed edge $xy$ of \G, which exist by Menger's theorem, and let $G'$ be the union of all these paths and $xy$. 

Let $T$ be a normal spanning tree of \G. Let $r$ be the root of $T$, and define the \defi{height} $h(v)$ of each $v\in V(G)=V(T)$ to be the distance between $v$ and $r$ in $T$. Apply the star-comb lemma to $T,U$ to obtain a subgraph $Z$ which is either a $U$-star or a $U$-comb. 

If $Z$ is a $U$-star, let $c$ be its centre. We claim that \g contains a $U$-double-star, with $c$ being one of its two infinite degree vertices. 
To prove this claim we introduce the following terms. Given an edge $e\in E(T)$, define the \defi{branch of $e$} to be the subgraph $B_e$ of \g induced by the vertices in the component of $T - e$ that does not contain $r$. Note that all neighbours of $B_e$ lie on a subpath $I_e$ of $T$ because $T$ is normal, and the top end-vertex $I^\mathrm{t}_e$ of $I_e$ lies in $e$. We choose $I_e$ to be minimal with these properties. Define the \defi{closed branch $\cls{B_e}$} as the subgraph  of \g induced by $B_e \cup I_e$. Easily, $\cls{B_e}$ is 2-connected since \g is. A \defi{$U$-ear} of $B_e$ is a path $E_e$ in $\cls{B_e}$ joining $I^\mathrm{t}_e$ to some other vertex of $I_e$, \st\ $E_e$ is otherwise disjoint from $I_e$ and contains a vertex of $U\cap B_e$. It is easy to find a $U$-ear for each  branch $B_e$ intersecting $U$ using the 2-connectedness of $\cls{B_e}$; indeed, note that $\cls{B_e}$ remains 2-connected after contracting $I_e - I^\mathrm{t}_e$. 

To prove our claim, note that there are infinitely many such branches $B_e$ with $e$ incident with $c$, and these branches are pairwise disjoint. Since there are only finitely many vertices below $c$ in $T$, we can find an infinite subset of those branches with $U$-ears (starting at $c$ and) ending at the same vertex $c'$. The union of those $U$-ears is the desired $U$-double-star. 


\comment{
	\begin{figure} 
\begin{center}
\end{center}
\caption{Obtaining a $U$-double-star or $U$-fan from a bouquet $B'$ of cycles, depending on whether $X'$ is a $U'$-star (left) or  $U'$-comb (right).} \label{figBouquet}
	\end{figure}
}

\medskip
This completes the case where $Z$ is a $U$-star, and we now proceed to the case  where  $Z$ is a $U$-comb. We further distinguish two cases according to whether the spine $R$ of $Z$ is \defi{dominated} in \g or not. Here, we say that $R$ is dominated, if there is a vertex $d$ sending infinitely many independent paths to $R$.

Suppose first that $R$ is dominated, by a vertex $d$ say. By extending or shortening $R$ as needed, we may assume that $R$ starts at $r$. (It follows that $d\in V(R)$, but we will not need this.) If $R \cap U$ is infinite then we immediately obtain a $U$-fan, with $d$ being the infinite degree vertex, and we are done. So assume $R \cap U$ is finite, and note that this means that there is an infinite set $\{B_{e_n}\}_\nin$ of pairwise disjoint branches each containing a tooth of $Z$; indeed, we can choose $e_n$ to be the $n$th edge of $Z$ incident with $R$ but not contained in it. Recall from above that each $\cls{B_{e_n}}$ contains a $U$-ear $E_{e_n}$, and note that each $E_{e_n}$ has both end-vertices on $R$. We distinguish two cases according to how these end-vertices are distributed along $R$:
\begin{itemize}
	\item[Case 1:] \label{R i}  No vertex of $R$ is incident with infinitely many $E_{e_n}$'s. In this case we can greedily choose an infinite  subsequence $\{E_{f_n}\}$ of  $\{E_{e_n}\}$ \st\ the subpaths $R_n$ of $R$ between the two end-vertices of each $E_{f_n}$ are pairwise disjoint. By replacing infinitely many of these subpaths $R_n$ by $E_{f_n}$, we can modify $R$ into a ray $R'$ \st\ $R'\cap U$ is infinite and $d$ still dominates $R'$ (\fig{figRdom}), hence obtaining a $U$-fan in \G. Indeed, we can greedily alternate between choosing a \pth{d}{R} and a $U$-ear $E_{f_n}$ to include in our $U$-fan; each time we choose a \pth{d}{R}, we delete the finitely many $U$-ears it meets, and conversely, each time we choose a $U$-ear $E_{f_n}$, we replace $R_n$ by $E_{f_n}$, and delete the finitely many \pths{d}{R}  it meets. Here, we assume that we have pre-selected an infinite family of pairwise independent  \pths{d}{R}, which we can since $d$ dominates $R$.
	\item[Case 2:] \label{R ii} Some $v \in V(R)$ is incident with infinitely many $E_{e_n}$'s. Note that $v$ must be the lower of the two end-vertices of these $E_{e_n}$'s, since the top end-vertex lies in $e_n$ by the definition of  $E_{e_n}$, with just one possible exception in case $v$ happens to be an end-vertex of some $e_n$. By the same argument, the other end-vertices of the  $E_{e_n}$'s incident with $v$ can be assumed to be pairwise distinct. Thus the union of all these $E_{e_n}$'s with a sub-ray of $R$ forms a $U$-fan in \G, with $v$ being the infinite degree vertex.
\end{itemize}

	\begin{figure} 
\begin{center}
\begin{overpic}[width=.7\linewidth]{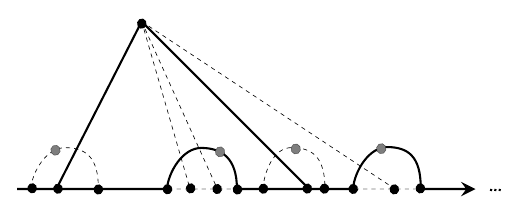} 
\put(23,36){$d$}
\put(84,6){$R$}
\put(3,12){$E_{f_0}$}
\put(28,12){$E_{f_1}$}
\put(79,12){$E_{f_4}$}
\end{overpic}
\end{center}
\caption{Obtaining a $U$-fan (bold lines) from a dominated comb $R$ of $T$ . The dashed lines represent paths that are not contained in the $U$-fan. The grey vertices represent elements of $U$.} \label{figRdom}
	\end{figure}

\medskip
This completes the case where $R$ is dominated, and we now proceed to the case where it is not. Again,  we may assume that $R$ starts at the root $r$. In this case we claim that 
\labtequ{claimH}{there is a \lf, 2-connected, subgraph $H$ of $G$ containing $R$ with infinite $H\cap U$.} 
To construct $H$, 
for each vertex $w\neq r\in V(R)$, let $P_w$ be a minimal path connecting the two components of $R - w$, which exists since \g is 2-connected. 
Suppose first that for some $w$, there are
such paths $P_w$ with end-vertices of arbitrarily large height. In this case we apply the star-comb lemma to their union, to obtain a $V(R)$-star or $V(R)$-comb $X$ in that subgraph. If $X$ was a $V(R)$-star then this would contradict the fact that $R$ is not dominated, and so $X$ is a $V(R)$-comb. If $(X \cup R)\cap U$ is infinite, then we just let $H:=  X \cup R$ and have achieved \eqref{claimH}. If not, then similarly to the case where $R$ is dominated, we can find infinitely many edges $e$ incident with $R$ \st\ their branch $B_e$ contains a tooth of $Z$ lying in $U$. For each such $B_e$, we pick a $U$-ear $E_e$, and add it to $X \cup R$ to obtain $H$.
Note that $H$ must be locally finite, because the $B_e$ are pairwise disjoint, and  if infinitely many ears $E_e$ share an end-vertex $v$ (on $R$), then $v$ dominates $R$, contradicting our assumption. 

If, on the contrary, no such $w$ exists, then we choose each $P_w$ so as to maximise the height of its top end-vertex on $R$. This choice of $P_w$ allows us to ensure that for $w\neq w' \in V(R)$, the paths $P_w, P_{w'}$ either coincide, or they are disjoint, or the top end-vertex of one of them coincides with the bottom end-vertex of the other; for otherwise we could find a path $P'$ in $P_w \cup P_{w'}$ that can serve as both $P_w$ and  $P_{w'}$. It follows that $H':= R \cup \bigcup_{w\in V(R)} P_w$ is \lf. Easily, $H'$ is also 2-connected. If $H'\cap U$ is infinite we set $H:= H'$ and have proved our claim. If it is finite, then again we can find infinitely many edges $e$ incident with $R$ \st\ their branch $B_e$ contains a tooth of $Z$ lying in $U$. For each such $B_e$, we pick a $U$-ear $E_e$, and add it to $H'$ to obtain $H$. Easily, $H$ is still  2-connected, and $H\cap U$ is infinite. As above, $H$ is \lf\ because $R$ is not dominated. This completes the proof of \eqref{claimH}.


We can now reduce our problem to the \lf\ case, by replacing \g by $H$. But we have handled the \lf\ case above, obtaining a $U$-ladder.
\end{proof}

\begin{problem} 
Is it possible to generalise \Tr{SC 2con} to $k$-connected graphs, obtaining a finite list of subdivisions of $k$-connected graphs as unavoidable structures?
\end{problem} 

Results of similar flavour have been obtained by Gollin \& Heuer \cite{GolHeuCha}.

\section{Final remarks} \label{final}

It would be interesting to find the excluded minors for the classes $\rmce{\Sig}$, $\rme{\mathrm{Planar}}$, $\rmce{\mathrm{Planar}}$ and $\rmece{\mathrm{Planar}}$, and this should be within reach with the above methods and a little bit more work. I suspect that $$\rmce{\Sig}= \rmce{\mathrm{Planar}} = \rmece{\mathrm{Planar}} =\forb{\omdot K_5, \omdot \Ktt, \bigvee K_5, \bigvee \Ktt}.$$ 
The first two equalities have been proved in \Prr{prop AP}. I also suspect that  
$\rme{\mathrm{Planar}}=\forb{ \ex{\Sig} \cup \{\prl{K_5}, \prl{\Ktt} \}}$, where $\prl{K}$ is obtained from a graph $K$ by replacing each edge $uv$ by infinitely many \pths{u}{v}\ of length 2.

\comment{
	\begin{question} \label{Q SigCE}
Is $\rmce{\Sig}=\forb{\omdot K_5, \omdot \Ktt, \bigvee K_5, \bigvee \Ktt}$?
\end{question}
}

\medskip

Let us say that a minor-closed class \cc\ of graphs is \defi{good}, if $\cc=\forb{X}$ for a finite set $X$ of (possibly infinite) graphs. A well-know conjecture of Thomas \cite{ThoWel} postulates that the countable graphs are \wqo\ under the minor relation. A positive answer would imply that all minor-closed classes of countable graphs are good, but as mentioned in the introduction, this seems out of reach at the moment.  Still, we could seek to extend the Graph Minor Theorem \cite{GMXX} by finding sufficient conditions for classes of infinite graphs to be good. The following questions suggest a possible direction, and the methods of this paper  could be helpful. For further questions in a similar vein see \cite{Universal}.

\begin{question} \label{Q good}
Suppose \cc\ is a good minor-closed class of countable graphs. Must each of $\rmv{\cc},\rme{\cc},\rmce{\cc},\rmece{\cc}$ be good?
\end{question}

We say that a class \cc\ of graphs is \defi{\cof}, if $\cc = \forb{S}$ for a set $S$ of finite graphs (which set can be chosen to be finite by the Graph Minor Theorem \cite{GMXX}). Note that a graph \g belongs to such a class \cc\ \iff\ every finite minor of \g does. \Qr{Q good} is open in general even if \cc\ is \cof, except that $\rmv{\cc}$ is covered by \Prr{prop PV} in this case. This papers provides some techniques for attacking it. In a similar spirit, one can ask whether the class of graphs admitting a finitary decomposition into graphs in $\cc$ is good whenever $\cc$ is good/\cof.

We say that a class \cc\ of graphs is \defi{\uncof} (Union of Nested \Cof\ classes), if there is a sequence \seq{C} of \cof\ classes \st\ $\cc=\bigcup_{\nin} C_n$ and $C_n \subseteq C_{n+1}$ holds \fe\ $\nin$. The classes studied in this paper ($\Sig,\frsE,\frsCE, \ope$, etc.) are easily seen to be \uncof. Our results support

\begin{conjecture} \label{Con good}
Every \uncof\ class of countable graphs is good.
\end{conjecture}

Another interesting example of an \uncof\ class \cc\ comprises the graphs \g of finite Colin de Verdi\`ere invariant $\mu(G)$, whereby for infinite \g we define  $\mu(G)$ to be the supremal $m$ \st\ every finite subgraph $H\subset G$ satisfies $\mu(H)\leq m$. Is this \cc\ good? Can we determine $\ex{\cc}$? 

Not every proper minor-closed class is \uncof. For example, $\forb{K_\omega}$ is not, because it contains the disjoint union of $K_n, \nin$, which no proper \cof\ class contains. Thus \Cnr{Con good} is weaker than Thomas' conjecture. Beware however that \Cnr{Con good} implies the Graph Minor Theorem: any minor-closed class of finite graphs is shown to be \uncof\ by letting $C_n$ be its sub-class comprising the elements with at most $n$ vertices.

\acknowledgement{I thank Nathan Bowler and Max Pitz for spotting a mistake in an earlier version of the paper. I thank the anonymous referees for proposing several substantial improvements.}

\comment{
	\begin{lemma} \label{lem}

\end{lemma}
\begin{proof}

\end{proof}

\begin{problem} \label{}

\end{problem} 
}


\bibliographystyle{plain}
\bibliography{collective}

\extras{

We generalise 
\begin{lemma}[\cite{halin74}] \label{LemHal}
Let \g be a countable, 2-connected, \lf\ graph, and $\chi$ an end of \G. Then \g contains two disjoint rays belonging to $\chi$.
\end{lemma}

... as follows: 

\begin{lemma} \label{LemLadFan}
Let \g be a countable, 2-connected, graph, and $\chi$ an end of \G. Then \g contains a subdivision of one of the graphs of \fig{} converging to $\chi$, i.e. a ladder or an infinite fan.
\end{lemma}
\begin{proof}
Pick a ray $R=r_0 r_1 \ldots$ of $\chi$. Since \g is 2-connected, \ti, \fe\ $i>0$, a path $P_n$ in \g joining the two components of $R-r_n$. Note that $X:= R \cup \bigcup_\nin P_n$ is 2-connected. We claim that it is possible to choose the $P_n$ so that $X - E(R)$ is a forest. Indeed, we can choose a spanning forest $F$ of $X - E(R)$, and replace each $P_n$ by a subpath $P'_n$ of its component in $F$. We hereby allow $P'_n$ to have different end-vertices than $P_n$, but insist that it connects the two components of $R-r_n$. 

We distinguish two cases: if $X$ is \lf, then \Lr{LemHal} yields two disjoint rays $X,Y$ equivalent with $R$. By greedily adding an infinite sequence of pairwise disjoint \pths{X}{Y}\ we thus obtain a ladder containing $X \cup Y$. If  $X$ is not \lf, then it has a vertex $v$ with infinitely many incident edges $e_0,e_1, \ldots$. Let $Q_i$ be a \pth{v}{R} containing $e_i$, which exists since each $e_i$ lies in some $P'_n$, which has both end-vertices on $R$. Recall that $X - E(R)$ is a forest, and therefore so is $\bigcup_{i\in \N} Q_i$. In particular the no two $Q_i$'s share a vertex other than $v$. Thus $R \cup \bigcup_{i\in \N} Q_i$ is a subdivision of a dominated ray.
\end{proof}

\begin{figure} 
\begin{center}
\end{center}
\caption{A \defi{ladder}, and a \defi{dominated ray}.} \label{figLadStar}
\end{figure}

}

\end{document}